\documentclass[11pt]{amsart}
\usepackage{amssymb}
\usepackage{amsfonts}
\usepackage{amsmath}
\usepackage{graphicx}
\usepackage{xcolor}
\usepackage{amsmath}
\usepackage{mathrsfs}
\usepackage{stmaryrd}
\usepackage{epsfig,color}
\usepackage{blindtext}
\usepackage{enumerate}
\usepackage{enumitem}
\usepackage{hyperref}
\usepackage{url}
\usepackage{bbm}
\usepackage{nicefrac,mathtools}
\usepackage{bm}   
\DeclareGraphicsExtensions{.pdf,.jpeg,.png}
\usepackage{epstopdf}
\usepackage{cancel} 
\usepackage[normalem]{ulem} 
\usepackage{verbatim} 

\usepackage{color}
\usepackage[msc-links, lite]{amsrefs}
\usepackage{geometry}
\newtheorem{theorem}{Theorem}[section]

\newtheorem{proposition}[theorem]{Proposition}
\newtheorem{lemma}[theorem]{Lemma}

\newtheorem{claim}[]{Claim}

\theoremstyle{definition}
\newtheorem{definition}[theorem]{Definition}
\theoremstyle{remark}
\newtheorem{remark}[theorem]{Remark}

\numberwithin{equation}{section}

\newcommand{\mf}{\mathbf}
\newcommand{\mb}{\mathbb}
\newcommand{\mc}{\mathcal}

\newcommand{\mk}{\mathfrak}
\newcommand{\mr}{\mathrm}

\newcommand{\eps}{\varepsilon}
\newcommand{\wti}{\widetilde}

\newcommand{\RR}{\mathbb R}

\newcommand{\Area}{\mathrm{Area}}
\newcommand{\diam}{\mathrm{diam}}

\newcommand{\Om}{\Omega}

\newcommand{\dist}{\operatorname{dist}}

\title[waist-width in non-compact 3-manifolds]{On the waist and width inequality in complete 3-manifolds with positive scalar curvature}

\date{\today}

\author{Yevgeny Liokumovich}
\address{Department of Mathematics, University of Toronto, 40 St George
	Street, Toronto, ON M5S 2E4, Canada}
\email{ylio@math.toronto.edu}

\author{Zhichao Wang}
\address{Department of Mathematics, University of British Columbia, Vancouver, BC V6T 1Z2, Canada}
\email{zhichao@math.ubc.ca}

\begin{document}

\begin{abstract}
We show that a complete non-compact 3-manifold with scalar curvature bounded below by a positive constant admits a singular foliation
by surfaces of controlled area and diameter.
\end{abstract}

\maketitle


\section{Introduction}

In \cite{Gro20}*{\S 3.10, Width/Waist Conjecture} Gromov conjectured that a complete manifold with scalar curvature $R \geq 6$ admits a singular foliation by surfaces of area and diameter
bounded by a universal constant. This conjecture was proved for compact 3-manifolds in \cite{LM20}. In this paper the result is generalized to complete non-compact 3-manifolds.

\begin{theorem} \label{main}
Let $M$ be a complete non-compact 3-manifold with scalar curvature $R \geq 6$. Then there exists a proper Morse function
$f: M \rightarrow \RR$, such that for every $t\in \mb R$ each connected component $\Gamma$ of $f^{-1}(t)$ satisfies
    \begin{gather*}
    \diam_M(\Gamma) < 18 \pi;\\
    \Area(\Gamma) \leq 33 \pi.
    \end{gather*}
\end{theorem}

For compact manifolds it was proved in \cite{LM20} that one can construct a foliation by surfaces of 
controlled area and diameter and genus at most $2$ (the maximal Heegaard genus of a compact 
3-manifold with positive scalar curvature). We conjecture that it should also be possible to prove a similar genus bound in Theorem \ref{main}.

\medskip
Recall that the {\em Urysohn $1$-width} of a Riemannian three-manifold $M$ is the infimum of real numbers $d\geq 0$ with the property that there exist a Morse function $f:M\to \mb R$ such that for all $t\in \mb R$, each connected component of $f^{-1}(t)$ has diameter bounded by $d$ from above. For three-manifolds with $R\geq 6$ and zero first Betti number, Gromov-Lawson \cite{Gromov-Lawson-diameter} proved an upper bound for Urysohn 1-width by considering  level sets of the distance function to a fixed point; in \cite{Gro20}*{\S 3.10, Property A}, Gromov proved the Urysohn 1-width upper bound without the topological conditions; in \cite{CL20} Chodosh and Li proved the Urysohn 1-width upper bound for $3$-dimensional mu-bubbles in $4$-manifolds with positive scalar curvature. For closed manifolds, the work of Maximo and the first author \cite{LM20} implies a new proof of the Urysohn 1-width bound. Inspired by their work, the second author with Zhu proved the Urysohn 1-width upper bound for mean convex domains (possibly non-compact) with non-negative Ricci curvature in \cite{Wang-Zhu-Urysohn}. Urysohn width bounds were used in \cite{CCL23} to obtain topological classification results for four- and five-manifolds with positive scalar curvature. As a direct corollary, Theorem \ref{main} gives a new proof of the Urysohn 1-width upper bound of three-manifolds with $R\geq 6$.

\subsection*{Idea of the proof}
Our proof relies on methods from \cite{LM20}, construction of $\mu$-bubbles and ideas of Gromov and Song
on construction of minimal hypersurfaces in non-compact manifolds (\cite{Gromov2014}, \cite{Song19}, see also \cite{ChLi20}).

We first cut the complete three-manifold $(M,g)$ along minimal surfaces with index less than or equal to 1. As a result, if these surfaces are countable (e.g. the metric $g$ is bumpy), then such a cutting process decomposes $M$ into countably many regions. Then by adapting an argument of Gromov \cite{Gromov2014} and Song \cite{Song19},
the interior of these regions does not contain any closed minimal surfaces. These regions are said to be {\em geometrically prime}, a generalization of a similar notion in the compact case \cite{LM20}.

\medskip
The key step is to decompose the geometrically prime regions into countably many small pieces with diameter and boundary area upper bounds; see Definition \ref{def:admissible regions} for the concept of {\em admissible regions}. 

For prime regions with an unstable boundary component, we will use $\mu$-bubble techniques and constrained minimization to construct finitely many surfaces with controlled size to decompose it as a union of an admissible region and finitely many regions with exactly one mean convex boundary component; see Lemma \ref{lem:convex extension}. By repeating this argument, we obtain the desired decomposition for geometrically prime regions with unstable boundary components.

For those geometrically prime regions whose boundaries are stable minimal surfaces (or the boundary is empty) we use arguments in \citelist{\cite{Gromov2014} \cite{Song19}} to obtain a two-sided surface with controlled size and non-vanishing mean curvature. Then for the mean convex region, the argument in the last paragraph gives the desired decomposition. For the mean concave region, we cut off an admissible region so that its complement consists of finitely many connected regions with the property that their boundary contains exactly one non-minimal component (the others are stable minimal surfaces). Moreover, each non-minimal boundary component has controlled size and non-vanishing mean curvature. Observe that we have dealt with these cases. Thus, by repeating the process, we obtain the desired decomposition.

\medskip
It remains to construct a Morse foliation for each admissible regions. To do this, we regard the region as a subdomain in geometrically  prime regions and then consider mean curvature flow starting on the mean convex surfaces. Note that the diameter upper bound may not be preserved in the flow.
To obtain the diameter bound we intersect level sets of the flow with admissible regions and make necessary surgeries
so that they have regularity of level sets of a Morse function, similarly to how it was done in \cite{LM20} and \cite{ChLi20}. Combining all of the foliations together finishes the proof.

\subsection*{Outline}
This paper is organized as follows.
In Section \ref{sec:preliminary}, we present some known results.
We first introduce the bumpy metric theorem in Section \ref{subsec:bumpy metric}; the area and diameter bounds for minimal surfaces with index less than or equal to one and $\mu$-bubbles are provided in Section \ref{subsec:surfaces with bounded size}. Finally, in Section \ref{subsec:geometrically prime}, we introduce the concept of geometrically prime regions and a dichotomy proposition by adapting arguments from \citelist{\cite{Gromov2014} \cite{Song19}}. Section \ref{sec:decomposition} is devoted to three decomposition lemmas. They are used to decompose geometrically prime regions into  compact domains of controlled diameter and boundary area.
In Section \ref{sec:foliation}, we construct foliation in each compact region and prove our main theorem.

\subsection*{Acknowledgement}  Y.L. would like to thank Davi Maximo for many useful conversations. Y.L. was partially supported by NSERC Discovery Grant and Sloan Fellowship.
Z.W. would like to thank Professor Jingyi Chen and Professor Ailana Fraser for their support and encouragement.
\section{Preliminary}\label{sec:preliminary}

\subsection{Bumpy metrics}\label{subsec:bumpy metric}
Let $M$ be a complete (possibly non-compact) Riemannian manifold. We say that
metric $g$ on $M$ is bumpy if there are no closed minimal submanifolds with non-zero Jacobi fields. By \cite{WhiteBumpy17}*{Theorem 2.1} (see also \cite{WhiteBumpy91}*{Theorem 2.1}) for any metric $g$ on $M$ there exists a sequence $g_i$ of smooth bumpy metrics on $M$ with $|g-g_i|_{C^3} \rightarrow 0$. Hence, without any loss of generality we can assume that our metric is bumpy. 
We will also use the bumpy metric theorem with partially fixed metrics.
\begin{theorem}[\cite{WhiteBumpy17}*{Theorem 3.1}]\label{thm:bumpy}
Let $(N,\gamma_0)$ be a complete Riemannian manifold. Let $U\subset N$ be an open set. Consider the class $\Gamma$ of smooth metrics $\gamma'$ on $N$ such that $\gamma'$ and $\gamma_0$ agree on $N\setminus U$. Then a generic metric $\gamma\in \Gamma$ has the following property: if $\Sigma$ is a closed, minimal immersed submanifold of $N$ and if each connected component of $\Sigma$ intersects $U$, then $\Sigma$ has no nontrivial Jacobi fields.
\end{theorem}

\subsection{Existence of surfaces with bounded size}\label{subsec:surfaces with bounded size}
In this part, we present the existence of surfaces with controlled area and diameter.  Those surfaces will be used to decompose the manifold into small pieces.

The classical candidate surfaces for decomposition are minimal surfaces with index less than or equal to 1. In the following result, the area estimates in the second item are from Marques-Neves \cite{Marques-Neves-Duke}*{Proposition A.1} and the diameter bounds are due to Schoen-Yau \cite{Schoen-Yau-diameter}*{Theorem 1} (see also \cite{Gromov-Lawson-diameter}).
\begin{lemma}[\cite{LM20}*{Theorem 2.1 and Corollary 2.4}]\label{lem:area and diameter}
Let $(M^3,g)$ be a complete three-manifold and $\Sigma\subset M$ be a two-sided closed embedded minimal surface. Suppose that $R\geq 6$ in a neighbourhood of $\Sigma$.
\begin{enumerate}
    \item If $\Sigma$ is stable, then 
    \[ \mr{diam}_\Sigma\leq \frac{2\pi}{3}, \quad \Area \leq \frac{4\pi}{3}.\]
    \item If $\Sigma$ has index one, then
    \[\mr{diam}_\Sigma\leq \frac{4\pi}{3}, \quad \Area \leq \frac{16\pi}{3}.\]
\end{enumerate}
\end{lemma}
Schoen-Yau \cite{Schoen-Yau-diameter}*{Theorem 1} also gives a local version for surfaces-with-boundary.
\begin{lemma}[\cite{LM20}*{Theorem 2.3}]\label{lem:diamter for surface with boundary}
  Let $(M^3,g)$ be a complete three-manifold and $\Sigma\subset M$ be an embedded stable minimal surface with boundary. Suppose that $R\geq 6$ in a neighbourhood of $\Sigma$. Then 
    \[  \mr{dist}_\Sigma(x,\partial \Sigma)\leq \frac{2\pi}{3}.\]
\end{lemma}

This implies the following Frankel property.
\begin{lemma}\label{lem:frankel}
Let $(N,\partial N,g)$ be a complete smooth Riemannian three-manifold with smooth boundary and $R\geq 6$. Suppose that each connected component of $\partial N$ is a closed surface with non-negative mean curvature. Suppose that there are two connected component $\Sigma_1$ and $\Sigma_2$ that are not stable minimal surfaces. Then there is a two-sided closed stable minimal surface in $(N\setminus \partial N,g)$.
\end{lemma}
\begin{proof}
Since $\Sigma_i$ ($i=1,2$) is not a stable minimal surface, then $\Sigma_i$ is either non-minimal or minimal but unstable. In either case, we can perturb it slightly to be strictly mean convex. Thus without loss of generality, we assume that $\Sigma_1$ and $\Sigma_2$ are mean convex. Let $\gamma$ be a curve connecting $\Sigma_1$ and $\Sigma_2$. 
Then by taking the area minimizer among all surface that is homologous to $\Sigma_1$, there is a stable two-sided minimal surface intersecting $\gamma$.
Such a minimizer must be compact because of Lemma \ref{lem:area and diameter}. This completes the proof of Lemma \ref{lem:frankel}.
\end{proof}

After cutting along minimal surfaces with index less than or equal to 1, we will use the $\mu$-bubble technique to subdivide a manifold with positive scalar curvature into submanifolds with boundary components of controlled size.

For a subset $\mc K$ of a Riemannian manifold $(M,g)$ and $r >0$, let 
\[N_r^g(\mc K) = \{ x \in M: \dist_{(M,g)}(x,\mc K) < r \}.\]
We will drop the superscript $g$ whenever it is clear which Riemannian metric is being used.
The following lemma is a consequence of arguments in  Zhu \cite{Zhu20}*{Theorem 1.1}, Gromov \cite{Gro20} and Chodosh-Li \cite{CL20}*{\S 6.2} (see also \cite{GroAspherical}).
\begin{lemma}\label{lem:mu bubble}
Let $N$ be a manifold with boundary and $R\geq 6$. Suppose that each connected component of $\partial N$ is a stable minimal surface. Then for each compact region $\Omega$, there exists a compact region $\hat \Omega$ with 
\[\Omega\subset \hat \Omega\subset N_{\frac{4}{3}\pi}(\Omega) \]
and for each connected component $\Gamma\subset\partial \hat \Omega\setminus \partial N$, it is one of the following cases:
\begin{enumerate}
    \item $\Gamma$ is diffeomorphic to $S^2$ and 
    \[ \mr{diam}_N\Gamma\leq \pi;\quad \Area(\Gamma)<2 \pi;\]
    \item $\Gamma$ is diffeomorphic to $\mb{RP}^2$ and 
    \[ \mr{diam}_N\Gamma\leq \pi;\quad \Area(\Gamma)< \pi;\]
    \item $\Gamma$ is a free boundary disc and 
    \[ \dist(x,\partial \Gamma)\leq \pi \text{ for } x\in\Gamma;\quad 
    \mr{diam}_N\Gamma\leq \frac{8}{3}\pi \quad \Area(\Gamma)< \pi;\]
\end{enumerate}
\end{lemma}
\begin{proof}
Let $\Omega_0$ be a small perturbation of
$N_{\frac{2}{3}\pi}(\Omega) = \{x \in N: \dist (x, \Omega)< \frac{2}{3}\pi \}$, so that $ \partial \Omega_0$ is smooth.
Let $d:N \rightarrow \RR$ denote a signed distance function from $\partial \Om_0$, $d(x) = \dist(x, \partial \Om_0) $
if $x \notin \Om_0$ and  $d(x) = - \dist(x, \partial \Om_0) $ if $x \in \Om_0$. Note that $d$ is Lipschitz. 
Denote the interior of $N_{\frac{4}{3}\pi}(\Omega) \setminus \Omega$ by $U$ and let 
$\rho: \overline{U} \rightarrow \RR $ be a $C^0$ perturbation of $d$ that is smooth and satisfies 
$\rho(x)= -\frac{2}{3}\pi$ for $x \in \partial \Omega$, 
$\rho(x)= \frac{2}{3}\pi$ for $x \in \partial N_{\frac{4}{3}\pi}(\Omega)$
 and $|\nabla \rho| < 1 + \eps$ for some small $\eps < \frac{1}{100}$. 
Define a smooth function $h$ on $U$ by
\[ h(x):=(1+\eps)\tan\Big[\frac{3}{4}\rho(x) \Big].\]
By \cite{CL20}*{Propositions 12 and 15} there exists a minimizer $\mc K$, which minimizes 
\[  \mc A^h(\mc K'):=\mc H^2(\partial \mc K'\setminus \partial N)-\int_{U} (\chi_{\mc K'}-\chi_{\Om_0})h \,\mr d\mc H^2\]
among all compact region $\mc K'$ satisfying
\[ \Omega\subset \mc K'\subset \{x\in N;\dist_N(x,\Omega)<\frac{4}{3}\pi\}.\]
Let $\Gamma$ be a connected component of $\partial \mc K\setminus \partial N$. Then $\Gamma$ is a smoothly embedded prescribed mean curvature surface with free boundary, with scalar mean curvature $H=h$ and meeting $\partial N$ orthogonally. The minimizing property gives that for all smooth functions $\phi$ defined on $\Gamma$,
\[ \int_\Gamma |\nabla \phi|^2-(|A|^2+\mr{Ric}(\mf n,\mf n)+\langle \mf n,\nabla h\rangle)\phi^2\,\mr d\mc H^2 \geq \int_{\partial \Gamma}A_{\partial N}(\mf n,\mf n)\phi^2\, \mr d\mc H^1,\]
where $A$ and $A_{\partial N}$ are the second fundamental forms of $\Gamma$ and $\partial N$; $\mf n$ is the unit normal vector field of $\Gamma$. Now let $\phi\equiv 1$. Recall that 
\[  |A|^2+\mr{Ric}(\mf n,\mf n)=\frac{1}{2}(|A|^2+|H|^2+R)-K_\Gamma.\]
Using $|A|^2 \geq \frac{1}{2} |H|^2$ (by Cauchy inequality) and $2 \pi \chi (\Gamma) = \int_{\Gamma} K_\Gamma d\mc H^2 -\int _{\partial \Gamma} A_{\partial N}(\mf n,\mf n) d\mc H^1$ (by Gauss-Bonnet and minimality of $\partial N$) we obtain 
\[2\pi \chi(\Gamma)\geq \int_\Gamma \frac{1}{2}|A|^2+\frac{1}{2}h^2+\langle\nabla  h,\mf n\rangle+\frac{1}{2}R\,\mr d\mc H^2\geq \int_{\Gamma} \frac{3}{4}h^2 -|\nabla h| +3 \geq (\frac{9}{4}-\eps) \Area(\Gamma)>2 \Area(\Gamma),\]
where $\chi(\Gamma)$ is the Euler characteristic of $\Gamma$; $H^{\partial N}$ is the mean curvature of $\partial N$.
If $\Gamma$ is non-orientable, then $\Gamma$ is diffeomorphic to $\mb {RP}^2$ and $\Area(\Gamma)< \pi$. If $\Gamma$ is orientable and closed, then $\Gamma$ is diffeomorphic to $S^2$ and $\Area(\Gamma)< 2\pi$. If $\Gamma$ is orientable and has non-empty boundary, then $\Gamma$ is a free boundary disc and $\Area(\Gamma)<\pi$.

It remains to prove the diameter upper bound.  
By applying the same argument as in \cite{LZ18}*{Proposition 2.2}, we conclude that 
\begin{itemize}
    \item if $\Gamma$ is closed, then $\mr{diam}_N\Gamma\leq  \pi$;
    \item if $\Gamma$ is a disc, then $\dist_N(x,\partial \Gamma)\leq \pi$ for $x\in\Gamma$.
\end{itemize}
Since each connected component of $ \partial N$ is a stable minimal surface, then by Lemma \ref{lem:area and diameter}, we have that its  diameter is bounded by $\frac{2}{3}\pi$. If $\Gamma$ is a disc it then follows that $\mr{diam}_N\Gamma \leq \frac{2}{3}\pi+ 2\pi = \frac{8}{3}\pi$.

Hence, Lemma \ref{lem:mu bubble} is proved.
\end{proof}

\subsection{Geometrically prime regions}\label{subsec:geometrically prime}
We will use the concept of {\em geometrically prime regions}, which was defined in \cite{LM20}*{Definition 2.5}. Note that the region here may be non-compact.
\begin{definition}
A region $N$ is {\em geometrically prime} if 
\begin{enumerate}
	\item $N\setminus \partial N$ does not contain any closed embedded minimal surfaces;
	\item each connected component of $\partial N$ is a compact minimal surface with index less than or equal to 1.
	\end{enumerate}	
	\end{definition}

\begin{remark}\label{rmk:two-sided separation}
Let $N$ be a geometrically prime region with $R\geq 6$. We conclude that each embedded closed two-sided surface $\Sigma\subset N\setminus \partial N$ separates $N$. Suppose not, then there exists an area minimizer $\Sigma$ in the homology class. Then $\Sigma$ is compact by Lemma \ref{lem:diamter for surface with boundary}. This gives a contradiction.
 \end{remark}

We use the dichotomy theorem about existence of a minimal surface or a mean convex foliation from \cite{Gromov2014} and \cite{Song19}*{Theorem 2.1} to prove the following result.
\begin{proposition}\label{prop:index one}
Let $N$ be a (possibly non-compact) three-manifold with (possibly empty) boundary, bumpy metric and $R\geq 6$. Suppose that all boundary components of
$N$ are stable minimal surfaces and that $N$ does not contain any closed one-sided embedded minimal surfaces with stable double covers or two-sided stable embedded minimal surfaces in its interior.
Then there are two possibilities:
\begin{enumerate}
    \item $N\setminus \partial N$ contains a closed two-sided minimal surface of index one;
    \item\label{item:no closed minimal surface} $N\setminus \partial N$ does not contain any closed embedded minimal surfaces.  Let $\Omega\subset N$ be a compact region with smooth boundary so that each connected component of $\partial N$ is either contained in $\partial \Omega$ or disjoint from it. In this case, there exists a connected, closed, two-sided, embedded $C^{1,1}$ surface $\Sigma\subset N\setminus(\Omega\setminus\partial \Omega)$ intersecting $\partial\Omega$ and having mean curvature vector pointing towards $\Omega$.
\end{enumerate}
\end{proposition}
\begin{proof}
If $N$ is compact, then we can apply min-max theory to obtain an index 1 two-sided embedded minimal surface in the interior of $N$ (see \citelist{\cite{Zhou-15-one-parameter}*{Theorem 1.1 and Remark 1.3} \cite{KMN-catenoid-estimate}*{Theorem 1.5}}). Now we assume that $N$ is non-compact. 

Suppose that $N\setminus \partial N$ does not admit an index $1$ two-sided embedded closed minimal surface in its interior. 

Let $\Omega\subset N$ be a compact region with smooth boundary so that each connected component of $\partial N$ is either contained in $\partial \Omega$ or disjoint from it. We will show that there exists a connected, closed, two-sided, embedded surface $\Sigma\subset N\setminus(\Omega\setminus\partial \Omega)$ intersecting $\partial\Omega$ and having mean curvature vector pointing towards $\Omega$.
We argue as in \cite{Song19}*{Theorem 2.1}.
Let $\mc K \supset N_{3 \pi}(\Omega)$ be a set with smooth boundary; we can assume that each connected component of 
$\partial N$ is either contained in $\partial \mc K$ or disjoint from it. Define a new metric $\wti g$ that coincides with $g$ outside of $N^g_{\frac{\pi}{2}}(\partial \mc K \setminus \partial N)$
and so that $\partial \mc K \setminus \partial N$ is mean convex with respect to  $\wti g$. By Theorem \ref{thm:bumpy}, we may assume that metric $\wti g$ is bumpy.
By the choice of $\wti g$ and Lemma \ref{lem:diamter for surface with boundary}, for each closed, connected two-sided minimal surface $\Gamma\subset (\mc K,\wti g)$ of index less than or equal to one, it is either disjoint from $N_r(\Omega)$ or disjoint from $N_{\frac{\pi}{2}}^g(\partial \mc K\setminus \partial N)$.

Consider mean curvature flow $\{ \partial \mc K_t \}$ starting on $\mc K_0 = \mc K$ in $(\mc K, \wti g)$. By \cite{White2000}*{Theorem 11.1}, there exist two possibilities: 
\begin{enumerate}[label=(\roman*)]
    \item\label{item:touch} there exists $t'>0$, such that $\partial\mc K_{t'} \cap \partial \Omega$ is non-empty;
    \item\label{item:minimal barrier} there exists set $B$, $\Omega \subset B \subset\mc K$, with $\partial B$ a smooth stable minimal surface in $(\mc K, \wti g)$.
\end{enumerate}

Consider the first possibility. Let $\hat t<t'$ be the first time of the flow touching $\Omega$. Now consider the constrained minimizing problem $\mc P(\Omega,\mc K;\wti g)$ (see \citelist{\cite{Lin-obstacle}\cite{ZhihanWang-obstacle}*{Section 2.4}} for the regularity): minimizing $\Area(\partial \Omega';\wti g)$ among all $\Omega'$ so that $\Omega\subset\Omega'\subset \mc K$ and $\partial \Omega'$ is smooth. Denote by $\hat \Omega$ a constrained minimizer of $\mc P(\Omega,\mc K;\wti g)$. Clearly, $\partial\hat \Omega\setminus \Omega$ is a stable minimal surface w.r.t. $\wti g$. Since $\partial \mc K_t\setminus \partial N$ is mean convex, then by Maximum Principle \cite{White-MP-2010}*{Theorem 1}, $\partial \mc K_t\setminus \partial N$ does not intersect $\partial\hat \Omega\setminus \Omega$ for all $t\in[0,\hat t)$. Moreover, we have that $ \hat \Omega\subset \mc K_{\hat t}$. In the next claim, we will prove that $\partial \hat \Omega$ intersects $\partial \Omega$. Denote by $\Sigma$ the connected component of $\partial \hat \Omega$ that intersect $\Omega$. Then this is the desired surface.
\begin{claim}\label{claim:minimizer intersects Omega}
$\partial \hat \Omega$ intersects $\partial \Omega$. 
\end{claim}
\begin{proof}[Proof of Claim \ref{claim:minimizer intersects Omega}]
Suppose not, then $\partial \hat \Omega$ is a stable minimal surface w.r.t. $\wti g$. Since $(N\setminus \partial N,g)$ has no two-sided stable minimal surface, then each connected component of $\partial \hat \Omega\setminus\partial N$ intersects $N_{\frac{\pi}{2}}^g(\partial \mc K\setminus \partial N)$. By Lemma \ref{lem:diamter for surface with boundary}, for each $x\in \partial \hat \Omega$,
\[  \dist_N(x, N_{\frac{\pi}{2}}^g(\partial \mc K\setminus \partial N)\leq \frac{2\pi}{3}.\]
It follows that 
\[  \dist_N(\partial\hat \Omega\setminus \partial N,\Omega)\geq 3\pi-\frac{\pi}{2}-\frac{2\pi}{3}>\pi,\]
which implies that $\hat \Omega\supset \{x\in N;\dist_N(x,\Omega)\leq \pi\}.$ This contradicts the fact of $\hat \Omega\subset \mc K_{\hat t}$. Then Claim \ref{claim:minimizer intersects Omega} is proved.
\end{proof}

Suppose we are in the second possibility. We can then apply the min-max construction to find a two-sided index $1$ minimal surface $S$. By the diameter estimate in Lemma \ref{lem:area and diameter}
and our assumption that there are no two-sided index $1$ minimal surfaces in $(N,g)$, we have that $S$ does not intersect $\Omega$. Let $\tilde{\mc K}^1$ denote a connected component of $\mc K \setminus S$
that contains $\Omega$. We can perturb $S \subset \partial \tilde{\mc K}^1$ to the inside using the first eigenfunction of the stability operator to obtain 
set $\mc K^1 \subset \tilde{\mc K}^1$ with mean convex boundary. We can now apply mean curvature flow to $\mc K^1$ and arrive at the same two possibilities \ref{item:touch} and \ref{item:minimal barrier} as before. If we are in situation \ref{item:touch},
we obtain mean convex surface $\Sigma$ as above. If we are in situation \ref{item:minimal barrier}, we apply the min-max construction again. Since the metric is bumpy, we can encounter possibility \ref{item:minimal barrier} only finitely many times.

So far we know that mean convex surface $\Sigma$ exists. 
Suppose $\Omega$ contains an embedded minimal surface $\Sigma'$ in its interior. By our assumption
$\Sigma'$ is either two-sided and unstable or one-sided with unstable double cover. Consider 
manifold with boundary $\Omega''$ obtained as metric completion of $\Omega \setminus \Sigma'$.
Then by Lemma \ref{lem:frankel} at least one connected component of $\Omega''$ (and hence $\Omega$)
contains a stable two-sided minimal surface in its interior, contradicting the assumption of the theorem.
Since this is true for all compact domains $\Omega$, we conclude that possibility \eqref{item:no closed minimal surface} of the theorem holds.
\end{proof}

The next theorem asserts that we can decompose $M$ into geometrically prime regions.
We will say that a countable collection of smooth connected 3-manifolds
with boundary $\{ U_i\}$ is a \emph{decomposition} of a 3-manifold $M$ if there exists 
a countable collection of disjoint embedded closed surfaces $\{ \Sigma_j\}$ in $M$,
such that there exists a Riemannian isometry from $M \setminus \bigcup_j \Sigma_j$ to the
union of interior of $\{ U_i\}$. (For example, $M = \mathbb{RP}^3$ and $\{ U \}$ is a decomposition of $M$ consisting of
one element $U$ equal to the metric completion of an open ball obtained by removing a non-trivial embedded $\mathbb{RP}^2 \subset M$.) By a slight abuse of notation we will identify the interiors of $U_i$ and the corresponding
isometric subsets of $M$.

\begin{theorem} \label{prime decomposition}
Let $M$ be a (possibly non-compact) manifold with (possibly empty) boundary, bumpy metric and $R\geq 6$.
There exists a decomposition $\{ N_i \}$ of $M$ into countably many geometrically prime submanifolds $N_i$. 
\end{theorem}

\begin{proof}
Since the metric is bumpy there exists a countable maximal collection of disjoint minimal surfaces $\{\Sigma_i\}$, such that  each $\Sigma_i$ satisfies one of the following three possibilities: $\Sigma_i$ two-sided and stable; $\Sigma_i$ is two-sided and index $1$; or $\Sigma_i$ is one-sided and has a stable double cover. We claim that $M \setminus \bigsqcup_i  \Sigma_i$ is the desired collection of geometrically prime submanifolds.
    
Let $N$ be the metric completion of a connected component of $M \setminus \bigsqcup_i  \Sigma_i$. Then we finish the proof by considering the following two possibilities.
\begin{itemize}
\item If all boundary component of $\partial N$ are stable, or if $M = N$ (and the set $\{\Sigma_i\}$ is empty),
then by Proposition \ref{prop:index one}, $N\setminus\partial N$ does not contain any closed minimal surface, and hence $N$
is geometrically prime. 
\item If $\partial N$ has a boundary component that is a minimal surface of index $1$, we now prove the conclusion by contradiction. Suppose on the contrary that there exists a closed embedded minimal surface $\Gamma$ in $N\setminus \partial N$. Then by assumption, $\Gamma$ is either unstable or one-sided and has an unstable double cover. In either case, we cut $N$ along $\Gamma$ and let $\wti N$ be the metric completion. Then $\partial \wti N$ contains two unstable minimal surfaces of $\wti N$. By Lemma \ref{lem:frankel}, $\wti N\setminus \partial \wti N$ contains a two-sided stable minimal surface. This gives a contradiction.
\end{itemize}
\end{proof}

\section{Decomposition of geometrically prime regions}\label{sec:decomposition}
Let $M$ be three dimensional Riemannian manifolds with $R\geq 6$. In this section, $N$ is always a geometrically prime region in $M$. 
We will divide such an $N$ into compact regions with controlled diameter and boundary area. More precisely, we will divide
$N$ into subsets satisfying the following definition.

	\begin{definition}\label{def:admissible regions}
Let $N \subset M$ be a geometrically prime region, $R\geq 6$. Then a compact region $\Omega\subset N$ is {\em admissible} if 
	\begin{enumerate}
	    \item $\partial \Omega$ is closed;
        \item each connected component $\Gamma'$ of $\partial \Omega$ satisfies
    \[ \mr{diam}_N\Gamma'\leq 6\pi; \]
    moreover,
    \[ \Area(\Gamma')\leq 4\pi \text{ if $\Gamma'$ is mean concave}; \quad \Area(\Gamma')\leq \frac{16}{3}\pi \text{ if $\Gamma'$ is mean convex};\]
    \item one connected component $\Gamma$ of $\partial\Omega$ is weakly mean convex and the others are weakly mean concave; moreover,
    \[ \Area(\partial \Omega)\leq \frac{32}{3}\pi;\]
    \item for each $x\in \Omega$ the distance to the mean convex component $\Gamma$ of $\partial \Omega$ satisfies
    \[\dist_N(x,\Gamma)< 6\pi.\]
    In particular, we have $\diam_N(\Omega) < 18 \pi$.
\end{enumerate}
	\end{definition}

\begin{proposition}\label{prop:decompose N without min}
Let $N$ be a geometrically prime region. Then $N$ can be decomposed into countably many admissible regions $\{\Omega_j\}$.
\end{proposition}

To prove Proposition \ref{prop:decompose N without min}, we will need three lemmas about existence of separating surfaces of
controlled size and mean curvature.

\subsection{Decomposition of mean convex regions}

The first lemma concerns existence of a mean convex surface of controlled size in a subset of $N$ with mean convex boundary. It gives that for a subdomain of a geometrically prime region with a mean convex boundary component (which is the unique non-minimal boundary component) one can cut off an admissible region near the mean convex boundary.

Let $(N,g)$ be a Riemannian manifold with smooth boundary. In this section, for any compact set $\mc K\subset N$, we let
\[N_r(\mc K) = \{ x \in N: \dist_N(x,\mc K) < r \}.\]
\begin{lemma}\label{lem:convex extension}
Let $N\subset M$ be a geometrically prime region. Let $\Sigma$ be a two-sided connected surface with non-vanishing mean curvature. Then $\Sigma$ separates $N$ into two connected components. Denote by $\Omega$ the one that has mean convex boundary, and suppose that $dist(x, \Sigma) > \frac{x}{\Sigma}$ for some $x \in \Omega$. Then there exists a compact connected domain $\hat \Omega \subset \Omega$ with $\Sigma \subset \partial \hat \Omega$ and satisfying the following.
\begin{enumerate}
\item $\partial \hat \Omega\setminus \partial\Omega$ is a closed and embedded surface in $N$ and has mean curvature vector pointing away from $\Sigma$;
    \item $\dist_N(\partial \hat\Omega\setminus \partial \Omega,\Sigma)\geq \frac{\pi}{4}$;
    \item $\sup\{\dist_N(x,\Sigma);x\in \hat \Omega\}< 3\pi $;
    \item Each connected component $\Gamma'\subset \partial\hat\Omega\setminus\Sigma$ satisfies $\Area(\Gamma')< \frac{7}{3}\pi$ and $\mr{diam}_N\Gamma'<6\pi$;
    \item $\Area(\partial \hat \Omega\setminus \Sigma)<\Area(\Sigma)$.
\end{enumerate}
In particular, $\hat \Omega$ is admissible provided that $\Area(\Sigma)\leq \frac{16}{3}\pi$ and $\mr{diam}_N\Sigma\leq 6\pi$.
\end{lemma}

\begin{proof}
Since $N$ is geometrically prime, then $N\setminus \partial N$ does not contain minimal surfaces. Note that $\Sigma= \partial \Omega\setminus \partial N$ is mean convex. 
Since $\Sigma$ is mean convex as a boundary component of $\Omega$, then by Lemma \ref{lem:frankel}, each connected component of $\Omega\cap \partial N$ is a stable minimal surface. Then by Lemma \ref{lem:area and diameter}, such a component has area bounded by $\frac{4}{3} \pi$ and diameter bounded by $\frac{2}{3}\pi$. 
By Lemma \ref{lem:mu bubble}, there exists a compact region  
\[\Omega'\subset \{x\in \Omega;\dist_N(x,\Sigma)<\frac{4}{3}\pi\}\]
with $\Sigma \subset \partial \Omega'$ and so that each connected component $\Gamma'\subset \partial \Omega'\setminus \partial \Omega$ is either an $S^2, \mb {RP}^2$ with
\[ \Area(\Gamma')<2 \pi, \quad \mr{diam}_N(\Gamma')\leq \pi,\]
or a free boundary disc with
\[ \Area(\Gamma')< \pi, \quad \mr{diam}_N(\Gamma')\leq \frac{8}{3}\pi.\]

Denote by $\{\mc B_j\}$ the connected components of $\partial \Omega'\setminus \partial\Omega$. Note that each $\mc B_j$ separates $\Omega$. Let $E_j$ denote  the metric completion of the connected component of $\Omega\setminus\mc B_j$ that does not contain $\Sigma$.

We can take a set $\wti \Omega$, $N_{\frac{2}{3}\pi + \frac{1}{4}\pi}(\Omega') \cap \Omega \subset \wti \Omega \subset N_{\frac{5}{3}\pi}(\Omega') \cap \Omega$, so that 
$\partial \wti \Omega \setminus \partial \Omega$ is a smooth surface. 
Moreover, we perturb $\wti\Omega$ so that each connected component of $\partial N$ is either contained in $\partial \wti\Omega$ or disjoint from it. Let $\mc C_j$ denote the union of connected 
components of $\partial\wti \Omega \setminus \partial \Omega$ that lie in $E_j$.
It follows that for all $\mc C_i$ and $\mc B_j$,
\begin{equation}\label{eq:dist BC}
    \dist_N(\mc C_j,\mc B_i)\geq \frac{2}{3}\pi+\frac{1}{4}\pi.
\end{equation} 

\begin{figure}[ht]
    \begin{center}
\def\svgwidth{1\columnwidth}
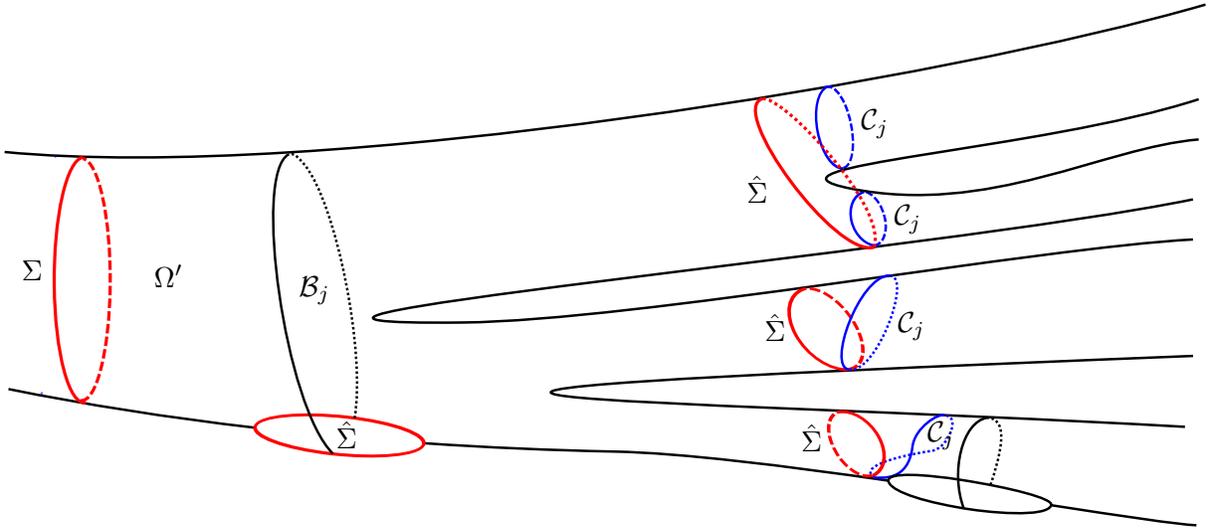
\caption{Cutting a mean convex region}
\label{fig:cutting near mean convex boundary}
\end{center}
\end{figure}

Let $\hat\Sigma$ be an area minimizer among all surfaces in $\wti \Omega$ that are homologous to $\Sigma$; see Figure \ref{fig:cutting near mean convex boundary}. Then $\hat \Sigma$ is a $C^{1,1}$ weakly mean concave surface (see \citelist{\cite{Lin-obstacle}\cite{ZhihanWang-obstacle}*{Section 2.4}} for the regularity). Moreover, $\hat\Sigma\setminus (\cup_j\mc C_j)$ is a stable minimal surface. Thus for each $p\in \hat\Sigma\setminus \partial N$, one has
\[ \dist_N(p,\cup_j \mc C_j)\leq \frac{2}{3}\pi.\]
This together with \eqref{eq:dist BC} implies that $\hat \Sigma$ does not intersect $\mc B_j$. Let $\hat\Omega$ be the connected component of $\Omega\setminus \hat \Sigma$ that contains $\Sigma$. Then for each connected component of $\partial\hat\Omega \setminus \Sigma$, it is either a closed embedded minimal surface or a mean concave surface intersecting $\cup \mc C_j$. Observe that for the first case, it is a connected component of $\partial N$. Thus it suffices to prove the diameter and area upper bounds for the second case.

Let $\Gamma'$ be a connected component of $\hat \Sigma$ that intersects $E_j$. Since $\Gamma'$ does not intersect $\mc B_j$ we have that $\Gamma'\subset E_j$. Moreover,
\[ \Gamma'\subset E_j\cap \wti\Omega.\]
It follows that 
\[ \mr{diam}_N\Gamma'< \mr{diam}_N\mc B_j+\frac{5}{3}\pi+\frac{5}{3}\pi\leq 6\pi.\]
The area upper bound for $\partial \wti\Omega \setminus \Sigma$ follows from $\Area(\hat \Sigma)<\Area(\Sigma)$.
The area upper bound for each connected component $\Gamma'$ of $
\hat \Sigma \setminus \partial \Omega $ follows by comparing 
the area of $\Gamma'$ to the area of a connected component $\mc B_{E_j}$ of $\partial E_j$ that separates $\Gamma'$ from $\Sigma$.
Recall that $\mc B_{E_j}$ is either a boundary component of $\partial \Omega'$ of area $< 2 \pi$, or
a union of a disc in $\partial \Omega' \setminus \partial \Omega$ and part of a component in $\partial \Omega$ of area
$<\pi + \frac{4}{3} \pi < \frac{7}{3}\pi$.

Moreover,
\[\dist_N(\Gamma',\Sigma)\geq \inf_i\dist_N(\mc C_i,\Sigma)-\frac{2}{3}\pi> \frac{\pi}{4}.\]
The third item comes from the fact of $\hat\Omega\subset\wti\Omega$. This completes the proof of Lemma \ref{lem:convex extension}.
\end{proof}

Given a geometrically prime region $N$ with one boundary component of Morse index $1$,
it is straightforward to decompose it into admissible regions
using Lemma \ref{lem:convex extension} repeatedly. But how can we decompose
a geometrically prime region that only has stable boundary components (or no boundary at all)?
The next two lemmas deal with this situation.

In the second lemma we consider a domain $\Omega\subset N$ with two non-minimal boundary components. Suppose that one of them is mean convex and the other is mean concave. We will show that $\Omega$ can be divided by a collection of closed surfaces into finitely many domains with mean convex boundaries and a compact region. Moreover, all of these surfaces have uniformly bounded area and diameter.
\begin{lemma}\label{lem:2 bdry extension}
Let $N$ be a geometrically prime region of $M$ with $R\geq 6$. Let $\Omega_1\subset \Omega_0 \subset N$ be two regions so that $\partial \Omega_i\setminus \partial N$ is a closed connected surface $\Sigma_i$ and let $\Omega = \Omega_0 \setminus \Omega_1$. Suppose that $\Sigma_i$ is mean convex as part of $\partial \Omega_i$ for $i=0,1$. Denote by
\[
d_0=\sup\{\dist_N(x,\Sigma_0);x\in\Sigma_1\}.
\]
Then there exists a compact connected domain $\hat \Omega\subset \Omega$ with $\partial \hat \Omega$ containing $\Sigma_0,\Sigma_1$ and satisfying the following.
\begin{enumerate}
    \item $\partial \hat \Omega\setminus \partial\Omega$ is a closed and embedded surface in $N$ and has mean curvature vector pointing away from $\Sigma_0$;
    \item $\sup\{\dist_N(x,\Sigma_0);x\in \hat\Omega\}< d_0+3\pi $;
    \item Each connected component $\Gamma'\subset \partial\hat\Omega\setminus(\Sigma_0\cup\Sigma_1)$ satisfies $\Area(\Gamma')< \frac{7}{3}\pi$ and $\mr{diam}_N\Gamma'<6\pi$;
    \item $\Area(\partial \hat\Omega\setminus (\Sigma_0\cup\Sigma_1))< \Area(\Sigma_0)$.
\end{enumerate}
In particular, $\hat \Omega$ is admissible provided that $d_0\leq 3\pi$ and $\Area(\Sigma_i)< 4\pi$, $\mr{diam}_N\Sigma_i\leq 6\pi$, $i=0,1$.
\end{lemma}
\begin{proof}
The argument is similar to Lemma \ref{lem:convex extension} with minor modifications.
Apply Lemma \ref{lem:mu bubble} to the compact set $N_{d_0}(\Sigma_0) \cap \Omega$; after intersecting with $\Omega$ we obtain a  compact region $\Omega'$, such that 
\[ N_{d_0}(\Sigma_0) \cap \Omega \subset \Omega' \subset N_{d_0+\frac{4\pi}{3}}(\Sigma_0) \cap\Omega,\]
so that each connected component $\Gamma'\subset \partial \Omega'\setminus \partial \Omega$ is either an $S^2, \mb {RP}^2$ with
\[ \Area(\Gamma')<2 \pi, \quad \mr{diam}_N(\Gamma')\leq \pi.\]
or a free boundary disc with
\[ \Area(\Gamma')< \pi, \quad \mr{diam}_N(\Gamma')\leq \frac{8}{3}\pi.\]
Denote by $\{\mc B_j\}$ the connected component of $\partial \Omega'\setminus \partial \Omega$. 
Note that $\mc B_j$ separates $\Omega_0$. Denote by $E_j$ the connected component of $\Omega_0\setminus \mc B_j$ that does not contain $\Sigma_0$. 
Observe that each connected component of $\partial N$ has diameter bounded by $\frac{2}{3}\pi$. We can take a set $\wti \Omega$, $N_{\frac{2}{3}\pi + \frac{1}{4}\pi}(\Omega') \cap \Omega_0 \subset \wti \Omega \subset N_{\frac{5}{3}\pi}(\Omega') \cap \Omega_0$, so that 
$\partial \wti \Omega \setminus \partial \Omega$ is a smooth surface (possibly with boundary in $\partial \Omega_0$).
Let 
\[\mc C_j=E_j\cap \partial \wti\Omega\setminus \partial \Omega_0, \quad  \wti{\mc C}=\Omega_1\cap \partial \wti \Omega\setminus \partial \Omega_0.\]
Note that $\mc C_j$ and $\wti{\mc C}$ may have boundary on $\partial \Omega_0$.
It follows that for all $\mc C_i$ and $\mc B_j$,
\begin{equation}\label{eq:dist BC and Sigma}
    \dist_N(\mc C_j,\mc B_i)\geq \frac{2}{3}\pi + \frac{1}{4}\pi, \quad \dist_N(\wti{\mc C},\Sigma_1)\geq \frac{2}{3}\pi + \frac{1}{4}\pi.
\end{equation} 
\begin{figure}[ht]
	\begin{center}
		\def\svgwidth{1\columnwidth}	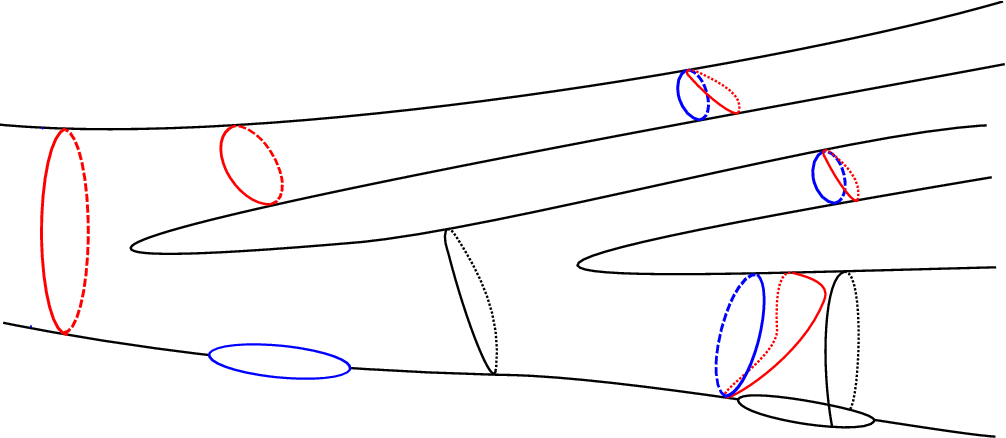
	\caption{Cutting the region bounded by two given surfaces}
		\label{fig:cutting region two boundary}
	\end{center}
\end{figure}

Let $\hat\Sigma$ be the area minimizer among all surfaces in $\wti \Omega$ that are homologous to $\Sigma_0$; see Figure \ref{fig:cutting region two boundary}. Then $\hat \Sigma$ is a $C^{1,1}$ weakly mean concave surface (see \citelist{\cite{Lin-obstacle}\cite{ZhihanWang-obstacle}*{Section 2.4}} for the regularity). Moreover, $\hat\Sigma\setminus (\cup_j\mc C_j\cup \wti{\mc C})$ is a stable minimal surface with respect to the normal vector field pointing away from $\Sigma_0$. Since $\hat \Sigma$ is a stable minimal surface outside $\cup\mc C_j\cup\mc C$, then by Lemma \ref{lem:diamter for surface with boundary}, for each $p\in \hat\Sigma$ lying in the interior of $\Omega_0$, one has
\[ \dist_N(p,\cup_j \mc C_j\cup\wti{\mc C})\leq \frac{2}{3}\pi.\]
Thus by \eqref{eq:dist BC and Sigma}, $\hat \Sigma\setminus \partial\Omega_0$ does not intersect $\Sigma_1$ or the interior of $\mc B_j$. Let $\hat\Omega$ be the connected component of $\Omega \setminus \hat \Sigma$ that contains $\Sigma_0$. Then for each connected component of $\partial\hat\Omega \setminus \partial \Omega$, it is either a closed embedded minimal surface or a weakly mean concave surface intersecting $\cup \mc C_j$. Observe that for the first case, it is a connected component of $\partial N$. The second item comes from the fact of $\hat\Omega\subset\wti\Omega$. The last item follows from that $\Area(\hat \Sigma)<\Area(\Sigma)$. The area and diameter estimates for each connected component of $\partial \hat \Omega\setminus (\Sigma_0\cup\Sigma_1)$ follows from the same argument in Lemma \ref{lem:convex extension}. This completes the proof of Lemma \ref{lem:2 bdry extension}.
\end{proof}

Let $\Sigma$ be a two-sided surface that separates $N$ into two subsets $\Omega$ and
$N \setminus \Omega$ and assume the mean curvature vector of $\Sigma$ is pointing 
away from $\Omega$.
Recall that by Proposition \ref{prop:index one}, there exists another two-sided surface in $\Omega $ with mean curvature vector pointing towards $\Sigma$. Assuming that $N$ has $R\geq 6$ we give an improvement: the new surface has controlled area and diameter. 
\begin{lemma}\label{lem:improved song}
Let $N$ be a geometrically prime region of $M$ with $R\geq 6$. 
Let $\Omega\subset N$ be a non-compact region so that $\Sigma:=\partial \Omega\setminus \partial N$ is a closed connected embedded mean concave surface and $\partial N \cap \partial \Omega$
is a collection of stable minimal surfaces.
Then there exists a connected closed embedded surface $\Sigma_1\subset\Omega$ such that 
\begin{enumerate}
    \item $\Sigma_1$ has mean curvature vector pointing towards $\Sigma$;
    \item $\mr{diam}_N\Sigma_1<5\pi$ and $\Area(\Sigma_1)<\frac{8\pi}{3}$;
    \item $\dist_N(\Sigma,\Sigma_1)\geq \frac{\pi}{3}$, $\sup\{\dist_N(x,\Sigma);x\in \Sigma_1\}< 3\pi.$
\end{enumerate}
\end{lemma}
\begin{proof}
Note that by Lemma \ref{lem:area and diameter}, each connected component of $\partial N$ satisfies
\begin{equation}\label{eq:stable partial N bound}
    \mr{diam}_N\leq \frac{2\pi}{3}, \quad \Area \leq \frac{4\pi}{3}.
    \end{equation}
By Lemma \ref{lem:mu bubble}, there exists a  compact region $\Omega'$,
\begin{equation}\label{eq:Omega' range}
N_{\frac{\pi}{3}}(\Sigma)\cap\Omega \subset  \Omega' \subset N_{\frac{5\pi}{3}}(\Sigma)\cap\Omega,
\end{equation}
so that each connected component $\Gamma'\subset \partial \Omega'\setminus \partial N$ is one of the following cases:
\begin{itemize}
    \item $\Gamma'$ is an $S^2$ or $\mb {RP}^2$ and 
    \[  \Area(\Gamma')<2\pi, \quad \mr{diam}_N(\Gamma')\leq \pi;\]
    \item $\Gamma'$ is a disc and 
    \[ \Area(\Gamma')<\pi, \quad \mr{dist}_{\Gamma'}(x,\partial \Gamma')\leq \pi, \quad \forall\,x\in \Gamma'.\]
\end{itemize}
Combining with \eqref{eq:stable partial N bound}, we have that for each connected component of $\partial\Omega'\setminus \Sigma$, it is either a smooth surface which is diffeomorphic to $\mb S^2$ or $\mb{RP}^2$ with 
\[  \Area(\Gamma')<2\pi, \quad \mr{diam}_N(\Gamma')\leq \pi;\]
or the union of a free boundary disc $D$ and part of a connected component of $\partial N$ with
\[ \Area<\pi+\frac{4\pi}{3}=\frac{7\pi}{3}, \quad \mr{diam}_N\leq 2\pi+\frac{2\pi}{3}=\frac{8\pi}{3}.\]
Note that $D$ has boundary on a minimal surface with diameter bounded by $\frac{2\pi}{3}$ (see Lemma \ref{lem:area and diameter}).
Then given $\delta>0$, by modifying non-smooth boundary of $\partial \Omega'$ slightly, we can take a region $\wti\Omega$ so that 
\begin{itemize}
    \item $\Omega'\subset \wti\Omega\subset\{x\in \Omega;\dist_N(x,\Omega')<\frac{2\pi}{3}+\delta\}$;
    \item $\partial \wti\Omega\setminus \partial N$ is a smooth surface;
    \item each connected component of $\partial \wti \Omega$ satisfies
    \[ \Area < \frac{8\pi}{3}, \quad \mr{diam}_N<3\pi.\] 
 \end{itemize}
Using \eqref{eq:Omega' range}, we can take $\delta>0$ sufficiently small so that
\begin{equation}\label{eq:wit Omega to Sigma}
    \sup_{x\in \wti\Omega}\dist_N(x,\Sigma)< \frac{7\pi}{3}.
\end{equation}
Denote by $\wti\Omega_0$ the connected component of $\wti\Omega$ that contains $\Sigma$. By Proposition \ref{prop:index one}, there exists a two-sided closed surface $\mc C\subset N_{20\pi}(\wti \Omega)\setminus N_{10\pi}(\wti\Omega)$ with mean curvature vector pointing towards $\wti\Omega$. Clearly, $\mc C$ does not intersect $\Sigma$. Thus, $\mc C$ is either contained in $\Omega$ or disjoint from $\Omega$. Recall that $\Sigma$ is mean concave and $N$ is geometrically prime. Then by Lemma \ref{lem:frankel}, we conclude that $\mc C\subset \Omega$. Now we take a large connected, compact region $\mc K$ satisfying  
\begin{itemize}
\item $\partial \mc K\setminus \partial N$ is a closed surface;
\item $\{x\in N;\dist_N(x,\wti\Omega)\leq \dist_N(\mc C,\wti \Omega)+\mr{diam}_N\mc C+100\pi\}\subset \mc K$.
\end{itemize}
		\begin{figure}[ht]
	\begin{center}
		\def\svgwidth{1\columnwidth}
		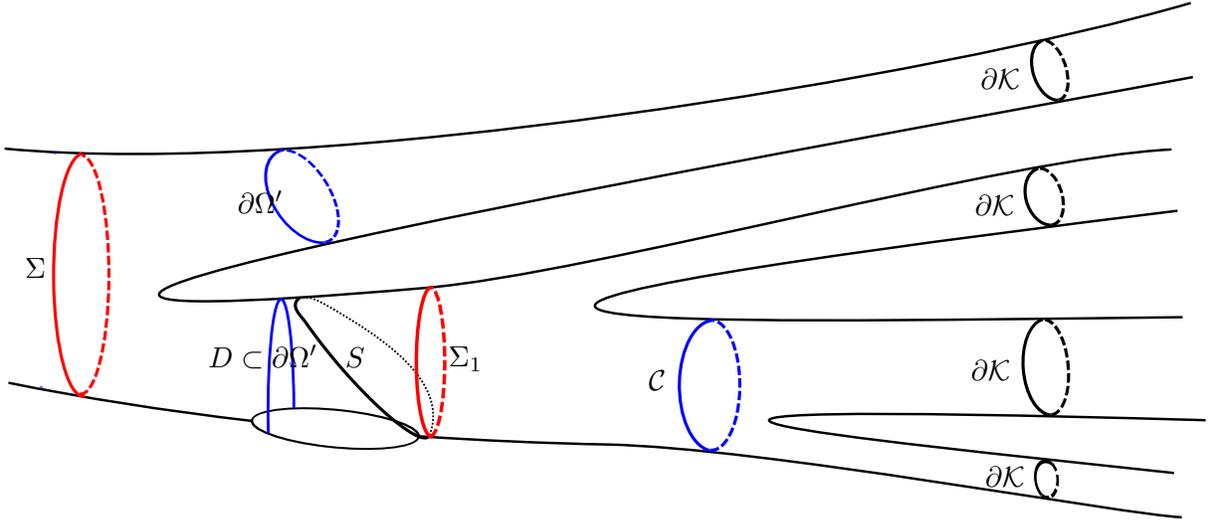
		\caption{Another mean convex surface with bounded size.}
		\label{fig:Another mean convex surface with bounded size}
	\end{center}
\end{figure}
Let $\wti{\mc K}$ be the metric completion of connected component of $\mc K\setminus (\mc C\cup\wti \Omega_0)$ that contains $\Sigma$ and $\mc C$. 
Denote by $S:=\partial \wti{\mc K}\cap \partial \wti{\Omega}_0$; see Figure \ref{fig:Another mean convex surface with bounded size}. 
\begin{claim}\label{claim:hat Sigma connected}
$S$ is connected.
\end{claim}
\begin{proof}[Proof of Claim \ref{claim:hat Sigma connected}]
Suppose not, let $S_1$ and $S_2$ be two connected components. Then they can be joined by a curve in $\wti{\mc K}$. 
By definition, $S_i\subset \wti\Omega_0$, $i=1,2$. Since $\wti\Omega_0$ is connected, then $S_1$ and $S_2$ can be joined by a curve in $\wti \Omega_0$. Thus we conclude that there is a simple closed curve intersecting $S_1$ with algebraic intersection number one.
This contradicts that $S_1$ is two-sided and separates $N$ by Remark \ref{rmk:two-sided separation}. 
\end{proof}
So far, we have proved that
$\partial \wti{\mc K}\setminus \partial N$ has at least two connected components: one is $\mc C$ (mean convex) and the other is $S\subset\partial \wti \Omega_0$. Now we consider the constrained minimizer among all closed surfaces that are homologous to $S$ in $\wti {\mc K}$.
Denote by $\hat \Sigma$ the minimizer. Clearly, $\hat \Sigma$ is a stable minimal surface (with boundary) in the interior of $\wti{\mc K}$.  
\begin{claim}\label{claim:minimizer intersect S}
    $\hat \Sigma$ intersects $S$.
\end{claim}
\begin{proof}[Proof of Claim \ref{claim:minimizer intersect S}]
Denote by $\mk C_0:=\partial \wti{\mc K}\setminus (\mc C\cup S\cup\partial N)$; that is, $\mk C_0$ is the non-minimal boundary of $\wti{\mc K}$ except $\mc C$ and $S$. Since $\wti{\mc K}$ has no closed minimal surfaces, then for each connected component of $\hat \Sigma$, it is either one of $\partial N$ or intersect $S\cup \mk C_0$. Suppose that $\hat \Sigma$ does not intersect $S$. Then $\hat \Sigma\setminus \partial N$ lies in a neighbourhood of $\mk C_0$ by Lemma \ref{lem:diamter for surface with boundary}. This contradicts $\hat \Sigma$ being in the homology class of $S$. 
\end{proof}
Let $\Sigma_1$ be the connected component of $\hat \Sigma$ that intersects $S$. Then $\Sigma_1$ intersects $\partial \wti\Omega\setminus \partial N$. Moreover, $\Sigma_1\setminus \wti\Omega$ is a stable minimal surface. 
Then by Lemma \ref{lem:diamter for surface with boundary}, for any $x,y\in \Sigma_1$,
\[ \dist_N(x,y)\leq \dist_N(x,S)+\dist_N(y,S)+\mr{diam}_NS<5\pi.\]
By the minimizing property of $\hat\Sigma$, we have
\[ \Area(\Sigma_1)\leq \Area(S)<\frac{8\pi}{3}.\]
It remains to prove the last item. Clearly,
\[ \dist_N(\Sigma,\Sigma_1)\geq \dist_N(\Sigma,S)\geq \frac{\pi}{3}.\]
By \eqref{eq:wit Omega to Sigma} and Lemma \ref{lem:diamter for surface with boundary}, for $x\in \Sigma_1$,
\[ \dist_N(x,\Sigma)\leq \dist_N(x,S)+\sup_{y\in S}\dist_N(y,\Sigma)< 3\pi.\]
This completes the proof of Lemma \ref{lem:improved song}.
\end{proof}

\subsection{Proof of decomposition of geometrically prime regions into admissible regions}

Now we can use the three lemmas above to prove Proposition \ref{prop:decompose N without min}.

\begin{proof}[Proof of Proposition \ref{prop:decompose N without min}]
We first prove the lemma for $N$ so that each connected component of $\partial N$ is a stable minimal surface. The proof can be divided into two steps. 

\medskip
{\noindent\bf Step I:} {\em  There exists a sequence of embedded surfaces $\{\Sigma_j\}_{j\geq 1}$ decomposing $N$ into $\{K_j\}_{j\geq 0}$ so that for $j=1,\cdots$,
\begin{itemize}
    \item $\partial K_{j-1}\setminus \partial N=\Sigma_{j-1}\cup\Sigma_{j}$ ($\Sigma_0:=\emptyset$);
    \item in $K_{j-1}$, $\Sigma_{j-1}$ is mean concave and $\Sigma_{j}$ is mean convex;
    \item $\mr{diam}_N\Sigma_j<5\pi$ and $\Area(\Sigma_j)< \frac{8\pi}{3}$;
    \item $\dist_N(\Sigma_j,\Sigma_{j+1})\geq \frac{\pi}{3}$ and  $\sup\{\dist_N(x,\Sigma_j);x\in\Sigma_{j+1}\}<3\pi$ for $j\geq 1$.
\end{itemize}
} 

\begin{proof}[Proof of Step I]
Lemma \ref{lem:improved song} applies to $\Sigma=\emptyset, \Omega=N$ to obtain a two-sided embedded closed mean convex surface $\Sigma_1$ so that 
\begin{itemize}
    \item $\Sigma_1$ is mean convex for a choice of normal vector fields;
    \item $\mr{diam}_N\Sigma_1< 5\pi$ and $\Area(\Sigma_1)< \frac{8\pi}{3}$.
\end{itemize}
Then $\Sigma_1$ separates $N$ into two regions by Remark \ref{rmk:two-sided separation}; denote by $K_0$ the one so that $\Sigma_1$ is mean convex. Then we construct $\Sigma_j$ inductively.

Suppose that we have $\Sigma_j$ and $K_{j-1}$. Note that $\Sigma_{i}$ is the common boundary of $K_{i-1}$ and $K_{i}$, $i=1,\cdots, j-1$. It follows that
\[ \partial (N\setminus \cup_{i=0}^{j-1}K_i)\setminus \partial N=\Sigma_{j}.\]
Then we apply Lemma \ref{lem:improved song} again for $\Omega=N\setminus \cup_{i=0}^{j-1}K_i$ and $\Sigma=\Sigma_j$ to get a two-sided embedded closed mean convex surface $\Sigma_{j+1}$ so that 
\begin{itemize}
    \item $\Sigma_{j+1}$ has mean curvature vector pointing towards $\Sigma_j$;
    \item $\mr{diam}_N\Sigma_{j+1}<5\pi$ and $\Area(\Sigma_{j+1})<\frac{8\pi}{3}$;
    \item $\dist_N(\Sigma_{j+1},\Sigma_j)\geq \frac{\pi}{3}$, $\sup\{\dist_N(x,\Sigma_j);x\in \Sigma_{j+1}\}< 3\pi$.
\end{itemize}
Hence Step I is finished by letting $K_{j}$ be the region bounded by $\Sigma_j$ and $\Sigma_{j+1}$.
\end{proof}

Then by Lemma \ref{lem:2 bdry extension} (with $d_0<3\pi$ therein), there exists a compact region $
\hat K_j\subset K_j$ so that 
\begin{itemize}
    \item $\partial \hat K_j\setminus \partial N$ is a closed embedded surface containing $\Sigma_j$ and $\Sigma_{j+1}$;
    \item $\partial \hat K_j\setminus \partial K_j$ is mean concave;
    \item $\sup\{\dist_N(x,\Sigma_{j+1});x\in \hat K_j\}< 3\pi+3\pi$;
    \item $\Area(\partial \hat K_j\setminus (\Sigma_j\cup\Sigma_{j+1}))<\Area(\Sigma_{j+1})$.
\end{itemize}
Here the third item together with the diameter bound of $\Sigma_{j+1}$ implies that $\mr{diam}_N(\hat K_j)<17\pi$; the last item gives that 
\[ \Area(\partial \hat K_j\setminus \Sigma_{j+1})<\Area(\Sigma_{j+1})+\Area(\Sigma_j)<\frac{16\pi}{3}.\]
Recall that by Remark \ref{rmk:two-sided separation}, each two-sided embedded surface separates $N$. Then $K_j\setminus \hat K_j$ consists of finitely many connected components. Moreover, for each connected component $K'$, $\partial K'\setminus \partial N$ is a connected mean convex surface $\Gamma'$ satisfying
\[ \mr{diam}_N\Gamma'<6\pi, \quad \Area(\Gamma')<\Area(\Sigma_{j+1})\leq \frac{8\pi}{3}.\]

Let $\Omega$ be one of the connected components of $K_j\setminus \hat K_j$ or $K_0$. Denote by $B=\partial \Omega\setminus \partial N$. Then $B$ is a connected mean convex surface with
\[ \mr{diam}_NB<6\pi, \quad \Area(B)< \frac{8\pi}{3}.\]
It suffices to divide $\Omega$ into compact regions as required. 

\medskip
{\noindent\bf Step II:} {\em For each $\Omega$ as above, it can be divided into countably many admissible compact regions $\{R_i\}$.}

\begin{proof}[Proof of Step II]
We prove it by induction. By Lemma \ref{lem:convex extension}, there exists a compact connected region $R_1$ containing $B$ so that 
\begin{itemize}
    \item $\partial R_1\setminus \partial \Omega$ is a closed embedded surface in $N$ and has mean curvature vector pointing away from $\Sigma$;
    \item $\dist_N(\partial R_1\setminus \partial \Omega,B)\geq \frac{\pi}{4}$, and $\sup\{\dist_N(x,B);x\in R_1\setminus \partial \Omega\}\leq 3\pi$;
    \item each connected component of $\partial R_1\setminus \partial \Omega$ has diameter less than $6\pi$;
    \item $\Area(\partial R_1\setminus B)<\Area(B)< \frac{8\pi}{3}$.
\end{itemize}
Observe that for each connected component $\Omega'$ of $\Omega\setminus R_1$, $B':=\partial \Omega'\setminus \partial \Omega$ is a connected mean convex surface and 
\[ \mr{diam}_NB'<6\pi, \quad \Area(B')<\frac{8\pi}{3}.\]
Thus we can repeat the above argument. This finishes Step II.
\end{proof}

We now consider the case when one connected component $\Sigma$ of $\partial N$ has index one.
Then by Lemma \ref{lem:area and diameter},
\[ \mr{diam}_N\Sigma\leq \frac{4}{3}\pi, \quad \Area(\Sigma)\leq \frac{16}{3}\pi.\]
Since $\Sigma$ is unstable, then there exists a neighborhood $\mc N_\delta$ of $\Sigma$ and a diffeomorphism $\phi:\Sigma\times[0,\delta)\rightarrow \mc N_\delta$ so that $\phi(\Sigma\times\{t\})$ has mean curvature vector pointing away from $\Sigma$.
It follows that 
\[ \Area(\phi(\Sigma\times\{\delta\}))<\frac{16}{3}\pi, \quad \mr{diam}_N<2\pi.\]
Thus $\mc N_\delta$ is an admissible region. Let $K$ be the metric completion of $N\setminus \mc N_\delta$. Then by Lemma \ref{lem:convex extension}, there exists an admissible region $\hat K$ containing $\phi(\Sigma\times\{\delta\}) $ so that the metric completion of $K\setminus \hat K$ consists of finitely many regions $\{\hat \Omega_j\}$ satisfying 
\begin{itemize}
    \item $\partial \hat \Omega_j\setminus \partial N$ is a closed embedded surface having mean curvature vector pointing away from $\phi(\Sigma\times\{\delta\})$;
    \item $\mr{diam}_N(\partial \hat \Omega_j\setminus \partial N)<6\pi$ and $\Area(\partial\hat \Omega_j\setminus \partial N)<\frac{7\pi}{3}< \frac{8\pi}{3}$.
\end{itemize}
Then applying the argument in Step II, each $\{\hat \Omega_j\}$ can be divided into countably many admissible regions. Therefore, we conclude that $N$ can be divided into admissible regions.
\end{proof}

\section{Foliation of admissible and prime regions and proof
of Theorem \ref{main}}\label{sec:foliation}
By modifying the mean curvature flow with surgery it was proved in \cite{LM20} that compact geometrically prime regions admit a singular foliation with bounded area. Similarly we will obtain foliations of admissible regions by singular surfaces with area bounds.

\begin{definition}
  Let $A \subset \mathbb{R}$ be a closed interval or a ray $[0, \infty)$.
Given a set $U \subset M$, we say that a family of closed surfaces $\{\Sigma_t \}_{t \in A}$ is a Morse foliation of $U$ if there exists a proper Morse function $f: U \rightarrow A $ with $f^{-1}(t) = \Sigma_t$, $t \in A$.
\end{definition}

Below we recall some basic results about Morse foliations of subsets of a Riemannian manifold.

\begin{lemma} \label{lem:admissible foliation1}
Let $M^3$ be a Riemannian manifold and $\Sigma \subset M$
be a smooth closed surface in $M$.
For every $\eps>0$ there exists $\delta(\Sigma, \eps)>0$,
such that $U = N_{\delta}(\Sigma)$ admits a Morse foliation $\{\Sigma_t\}_{t \in [0,1]}$  with $\Sigma_0 = \partial U$,
$\Sigma_1 = \{p\}$ for a point $p$ in the interior of $U$.
Moreover, we have $\Area(\Sigma_t) \leq 2 \Area(\Sigma)+ \eps$.
\end{lemma}

\begin{proof}
Observe that for an arbitrarily small $r>0$ we can choose
$\delta>0$ sufficiently small, so that $U$ is an interval bundle over $\Sigma$, $\pi: U \rightarrow \Sigma$, and
 for every $p \in \Sigma$ we have that $\pi^{-1}(B_r(p))$ is
$(1+ \frac{\sqrt{\eps}}{1000})$-bilipschitz diffeomorphic 
to $B_r^{\mathbb{R}^2}(0) \times [-\delta, \delta]$. 
In the construction below we will use local coordinates $(x,s)$, $x \in \Sigma$,
$t \in [-\delta, \delta]$, on $U$.

Fix a Morse function $g: \Sigma \rightarrow [0,1]$
and consider a family of surfaces with boundary $\{ \Sigma' _t \}_{t \in [0,1]}$
defined by $\Sigma'_t = (\pi \circ g)^{-1}(t)$.
Observe that for sufficiently small $\delta>0$ we have
that $\Area (\Sigma'_t) < \frac{\eps}{4}$. 
Also, we can slightly modify 
the family near critical points of $g$ to turn it into a family of level sets of a Morse
function on $U$.
Indeed, if $x$ is a critical point of $g$ of index $0$ or $2$ with critical value $g(x) = t'$, then
in the neighbourhood of $\pi^{-1}(x)$
surfaces $\Sigma'_t$ look like a family of cylinders $\{C_t \times [-\delta, \delta]\}$.
It is straightforward to modify this family to a family of concentric
spheres around point $(x,0)$ as $t \rightarrow t'$. Similarly, if $x$ is a critical point of $g$ of index $1$,
then surfaces $\Sigma'_t$ near $(x,0)$ look like a family $\{H_t \times [-\delta, \delta]\}$ 
where $H_t$'s are hyperbolas. Creating a small neck that connects two connected components of
$H_t \times [-\delta, \delta]$ and opening up the neck we modify this family of surfaces in the 
neighbourhood of $\pi^{-1}(x)$ to a family of level sets of Morse function with a cricial point of Morse index $1$.

After these modifications we obtain a Morse function $f:U \rightarrow [0,1] $
with critical points in the interior of $U$ and the area of fibers bounded by 
$\frac{\eps}{2}$. We can now apply \cite{ChLi20}*{Lemma 4.1}
to obtain the desired Morse foliation of $U$.
\end{proof}

\begin{lemma}  \label{lem:admissible foliation2}
        Let $\Omega$ be an admissible region.
Suppose $S_0$ is a closed (possibly empty) subset of 
$\partial \Omega $ and $S_1 = \partial \Omega \setminus S_0$. 
Then for every $\eps>0$ there exists a Morse foliation $\{ \Sigma_t \}$ of $\Omega$, such that
\begin{itemize}
    \item if $S_0$ is empty, then $\Sigma_0 = \{ p\}$ is a point in the 
    interior of $\Omega$ and $\Sigma_1 = S_1$;
    \item if $S_0$ is non-empty, then $\Sigma_0 = S_0$ and $\Sigma_1 = S_1$;
\end{itemize}
and $\Area(\Sigma_t) \leq 3\Area(\partial \Omega) + \eps < 33 \pi$.
\end{lemma}

\begin{proof}
    Let $U$ denote a small tubular neighbourhood of 
    $\partial \Omega$ in $\Omega$ and $\Omega' = \Omega \setminus U$.
    Note that each connected component $U_i$ of $U$ can be foliated by
    a family of surfaces $\{ \Gamma^i_t\}$, where $\Gamma^i_0= \Gamma^i$ is a connected component of
    $\partial \Omega$ and $\Area(\Gamma^i_t) \leq  \Area(\Gamma^i) + \delta $
    for $\delta>0$ that can be chosen arbitrarily small by making neighbourhood $U$
    small enough.

    Consider $\Omega$ as a subset of $M$ and define Mean Curvatue Flow with surgery starting on the mean convex
    boundary component $S$ of $\partial \Omega$. As in \cite{LM20} we can
    modify this family of surfaces to obtain a Morse function 
    $f': \Omega \rightarrow \mathbb{R}$ with the area of fibers at most 
    $\Area (S) + \delta$. (Note that some of the fibers
    may not be closed surfaces and may have boundary in $\partial \Omega \setminus S$.)
    We can also assume that $f'$ is a Morse function
    when restricted to $\Omega'$ and $\partial \Omega'$.
    By \cite{ChLi20}*{Lemma 4.1} there exists a Morse foliation $\{\Sigma'_t\}$ of $\Omega'$
    starting at a point $\Sigma'_0 = \{p\}$ in the interior of $\Omega'$, ending on $\Sigma'_1 = \partial \Omega'$
    and satisfying $\Area ( \Sigma'_t) \leq \Area (S) + \Area (\partial \Omega') + 2\delta $.

If $S_0$ is empty then we can extend  $\{\Sigma'_t\}$ to $U$ to obtain the desired Morse foliation of $\Omega$.

Suppose $S_0$ is not empty. Let $\tilde U$ denote the union of connected components
of $U$ that intersect $S_0$ and $\{ \tilde \Gamma_t^0 \}_{t \in [0,1]}$ denote the foliation of 
$\tilde U$ by surfaces isotopic to $S_0$, as defined above.
Let $\{ \tilde \Gamma_t^1 \}_{t \in [0,1]}$ denote the corresponding foliation
of $U \setminus \tilde U$.
Let $V$ denote a subset of $\tilde U$ bounded by $\tilde \Gamma_{\frac{1}{2}}^0$
and $\tilde \Gamma_{1}^0$.
By Lemma \ref{lem:admissible foliation1} we can define 
a Morse foliation $\{\tilde \Sigma_t \}_{t \in [0,1]}$ of
$V$ that starts on $\tilde \Gamma_{\frac{1}{2}}^0 \cup \tilde \Gamma_{1}^0$ and terminates
at a point in the interior of $V$.

We can now define the desired Morse foliation of $\Omega$ by setting 
\begin{equation*}
\Sigma_t=\left\{
\begin{aligned}
  & \tilde \Gamma_t^0& \text{ for $t \in [0,\frac{1}{4})$},\\
&\Sigma'_{4t-1} \cup \tilde \Gamma_t^0& \text{ for $t \in [\frac{1}{4}, \frac{1}{2})$},\\
&\tilde \Sigma_{4t-2} \cup \tilde \Gamma_{2t-1}^1& \text{ for $t \in [\frac{1}{2}, \frac{3}{4})$},\\
&  \tilde \Gamma_{2t-1}^1& \text{  for $t \in [\frac{3}{4}, 1]$}.
  \end{aligned}\right.
\end{equation*}
It is straightforward to check that with the definition above we have $\Sigma_0 = S_0$, $\Sigma_1= S_1$
and the area bound is satisfied.
\end{proof}

\begin{proof}[Proof of Theorem \ref{main}]
Let $M$ be a complete $3$-manifold with scalar curvature $R \geq 6$.
We perturb the metric on $M$ to a bumpy metric. We will construct the
desired Morse foliation for the perturbed metric and note that the diameter 
and area bounds for surfaces in the folation will be worse by at most $\eps$
for the original metric, where $\eps>0$ depends on the perturbation and
can be chosen arbitrarily small.

By Theorem \ref{prime decomposition} and Proposition \ref{prop:decompose N without min}
there exists a decomposition of $M$ into countably many admissible regions $\mathcal{U} = \{\Omega_i\}$.
Define graph $G$, where each vertex $V_{\Omega_i}$ corresponds to an admissible region
$\Omega_i$ and two vertices $V_{\Omega_i}$ and $V_{\Omega_j}$, $i \neq j$, are connected by an edge if
and only if the inclusions of $\Omega_i$ and $\Omega_j$ share a common boundary component in $M$.

Choose an admissible region $\Omega_0 \in \mathcal{U}$. 
Let $i_{\Omega_0}:\Omega_0 \rightarrow M$ denote the inclusion of $\Omega_0$
into $M$ that is a Riemannian isometry in the interior of $\Omega_0$.
Let $S_1(\Omega_0)$ denote the image of the boundary components of $\Omega_0$ 
on which $i_{\Omega_0}$ is a 2-to-1 map (for example, if 
$i_{\Omega_0}$ sends the interior of
$\Omega_0$ onto the interior of $\mathbb{RP}^3$ with some balls and nontrivial $\mathbb{RP}^2$
removed).
By Lemma \ref{lem:admissible foliation2}, for a sufficiently small $\delta>0$,
there exists a Morse function $f: \Omega_0 \setminus N_{\delta}(\partial \Omega_0) \rightarrow [0,1-\delta]$, such that $f^{-1}(0)$ is a point in the interior of $\Omega_0$ and $f^{-1}(1) = \partial (\Omega_0 \setminus N_{\delta}(\partial \Omega_0))$ and the fibers of $f$
satisfy the desired area and diameter bounds. We can extend $f$ to a Morse foliation of
$i_{\Omega_0}(\Omega_0)$, $f:i_{\Omega_0}(\Omega_0) \rightarrow [0,1]$, by applying Lemma \ref{lem:admissible foliation1} in the $\delta$-neighbourhood
of $S_1(\Omega_0)$ and foliating the rest of $i_{\Omega_0}(N_\delta(\partial \Omega_0))$ by level sets of the distance function.

Let $V_n$ denote the collection of all vertices in graph $G$ that are
at a graph distance $n$ from vertex $V_{\Omega_0}$ and let $U_n$ denote closure of the
union of the corresponding regions in $M$. Suppose, by induction,
that we have defined a Morse function $f: \bigcup_{k=0}^{n-1} U_k \rightarrow [0, n]$
with fibers of the function satisfying the desired area and diameter bounds and, moreover,
we have that $f(x) = n$ for $x \in \partial \bigcup_{k=0}^{n-1} U_k$.

For each admissible region $\Omega'$ corresponding to a vertex in $V_n$ we decompose $\partial \Omega'$
into four closed subsets. 
Let $S_0(\Omega') = \partial i_{\Omega'}({\Omega'}) \cap \partial U_{n-1}$, 
$S_1(\Omega')$ denote the connected components of $i_{\Omega'}( \partial \Omega')$
on which $i_{\Omega'}$ is a double covering,
$S_2(\Omega') = \partial i_{\Omega'}({\Omega'})\cap \partial (U_n \setminus i_{\Omega'}({\Omega'}))$,
and
$S_3(\Omega') = \partial i_{\Omega'}({\Omega'}) \setminus (S_0(\Omega') \cup S_1(\Omega') \cup S_2(\Omega'))$.
Note that $S_1$, $S_2$ and $S_3$ may be empty.

For a sufficiently small $\delta>0$ we can apply Lemma \ref{lem:admissible foliation2}
to define $f$ on $\bigcup i_\Omega(\Omega) \setminus N_{\delta}(S_1(\Omega) \cup S_2(\Omega) \cup S_3(\Omega))$, where the union is over all sets $\Omega \in \mathcal{U}$ correspond to
vertices in $V_n$, mapping it onto
$[n, n+1 - \delta]$, so that it agrees with the inductive definition of $f$ on $f^{-1}(n)$ and
$f^{-1}(n+1 - \delta) = \partial \bigcup N_{\delta}(S_1(\Omega) \cup S_2(\Omega) \cup S_3(\Omega)) $. 
We can then apply Lemma \ref{lem:admissible foliation1} to extend $f$ on 
$\bigcup N_\delta(S_1(\Omega) \cup S_2(\Omega)) $, mapping it to $[n+1 - \delta, n+1]$.
Finally, the set $\bigcup N_\delta(S_3(\Omega)) \cap \Omega$ can be foliated by surfaces
isotopic to $\bigcup S_3(\Omega)$, giving us a Morse foliation with $f= n+1$ on $\partial \bigcup_{k=0}^{n} U_k$. This finishes the inductive construction of $f$.

It follows from the definition of admissible regions
and the area and diameter estimates above
that for each connected component $\Sigma$ of $f^{-1}(x)$ we have $\Area(\Sigma) < 33 \pi$.
Since each surface is contained inside an admissible region we have that
the extrinsic diameter satisfies $\diam_M\Sigma < 18 \pi$.
\end{proof}

\bibliographystyle{amsalpha}
\bibliography{width}
\end{document}